\documentclass[12pt]{article} % use larger type; default would be 10pt

\usepackage[utf8]{inputenc} % set input encoding (not needed with XeLaTeX)

\usepackage{geometry} % to change the page dimensions
\usepackage{graphicx} % support the \includegraphics command and options
\usepackage{amsthm}

\usepackage{booktabs} % for much better looking tables
\usepackage{array} % for better arrays (eg matrices) in maths
\usepackage{paralist} % very flexible & customisable lists (eg. enumerate/itemize, etc.)
\usepackage{verbatim} % adds environment for commenting out blocks of text & for better verbatim
\usepackage{subfig} % make it possible to include more than one captioned figure/table in a single float
\usepackage[all]{xy}%make commutative graph
\usepackage{amsfonts}
\usepackage{amssymb}
\usepackage{amsmath}
\numberwithin{equation}{section}
\usepackage{fancyhdr} % This should be set AFTER setting up the page geometry
\newtheorem{thm}{Theorem}[section]
\newtheorem{lemma}{Lemma}[section]

\newtheorem{remark}{Remark}[section]

\newtheorem*{proof of theorem1.1}{proof of theorem1.1}
\newtheorem*{proof of theorem1.2}{proof of theorem1.2}
\usepackage{sectsty}
\allsectionsfont{\sffamily\mdseries\upshape} % (See the fntguide.pdf for font help)
\usepackage[nottoc,notlof,notlot]{tocbibind} % Put the bibliography in the ToC
\usepackage[titles,subfigure]{tocloft} % Alter the style of the Table of Contents

\date{}
\begin{document}
\title{\large\textbf {Normalized solutions to Schr\"{o}dinger systems with linear and nonlinear couplings}\thanks{Supported by National Natural Science Foundation of China(11771428, 12031015, 12026217), Email:yuanzhaoyang17@mails.ucas.ac.cn(Zhaoyang Yun); zzt@math.ac.cn(Zhitao Zhang)}}
\author{Zhaoyang Yun$^{a,b}$~~Zhitao Zhang$^{a,b,}\thanks{Corresponding author}$}
\maketitle
\centerline{\small $^{a}$ HLM, Academy of Mathematics and Systems Science, the Chinese}
\centerline{\small Academy of Sciences, Beijing 100190;}
 \centerline{\small $^{b}$  School of Mathematical Sciences, University of Chinese}
\centerline{\small Academy of Sciences, Beijing 100049, P.R. China}
% \centerline{\small Email:yuanzhaoyang17@mails.ucas.ac.cn(Zhaoyang Yun); \\ zzt@math.ac.cn(Zhitao Zhang)}
\begin{abstract}

In this paper, we study important Schr\"{o}dinger systems with linear and nonlinear couplings
\begin{equation}\label{eq:diricichlet}
\begin{cases}
-\Delta u_1-\lambda_1 u_1=\mu_1 |u_1|^{p_1-2}u_1+r_1\beta |u_1|^{r_1-2}u_1|u_2|^{r_2}+\kappa (x)u_2~\hbox{in}~\mathbb{R}^N,\\
-\Delta u_2-\lambda_2 u_2=\mu_2 |u_2|^{p_2-2}u_2+r_2\beta |u_1|^{r_1}|u_2|^{r_2-2}u_2+\kappa (x)u_1~ \hbox{in}~\mathbb{R}^N,\\ u_1\in H^1(\mathbb{R}^N), u_2\in H^1(\mathbb{R}^N),\nonumber
\end{cases}
\end{equation}
with the condition
$$\int_{\mathbb{R}^N} u_1^2=a_1^2, \int_{\mathbb{R}^N} u_2^2=a_2^2,$$
where $N\geq 2$, $\mu_1,\mu_2,a_1,a_2>0$, $\beta\in\mathbb{R}$, $2<p_1,p_2<2^*$, $2<r_1+r_2<2^*$, $\kappa(x)\in L^{\infty}(\mathbb{R}^N)$ with fixed sign and $\lambda_1,\lambda_2$ are Lagrangian multipliers. We use Ekland variational principle to prove this system has a normalized radially symmetric solution  for $L^2-$subcritical case when $N\geq 2$,
and use minimax method to prove this system has a normalized radially symmetric  positive solution  for $L^2-$supercritical case when $N=3$, $p_1=p_2=4,\ r_1=r_2=2$.
\end{abstract}
\textbf{Keywords: }Nonlinear Schr\"{o}dinger systems; Normalized solutions; Ekland variational principle; Minimax principle.\\
\textbf{AMS Subject Classification(2010): }35J15, 35J47, 35J57
\section{Introduction}
Schr\"{o}dinger systems of the form  which are related with Bose-Einstein condensates
\begin{equation}\label{eq:diricichlet}
\begin{cases}
-\Delta u_1-\lambda_1 u_1=f_1(u_1)+\partial_1 F(u_1,u_2)\ \textup{in}\ \mathbb{R}^N,\\
-\Delta u_2-\lambda_2 u_2=f_2(u_2)+\partial_2 F(u_1,u_2)\ \textup{in}\ \mathbb{R}^N,\\
u_1\in H^1(\mathbb{R}^N), u_2\in H^1(\mathbb{R}^N),
\end{cases}
\end{equation}
where $N\geq 2$, have been concerned by many mathematicians in recent years. One motivation driving the search for (1.1) is to find the solutions of the time-dependence system of coupled nonlinear Schr\"{o}dinger equations.\par
 In \cite{B-2} T. Bartsch and L. Jeanjean studied the case when $f_1(u_1)=\mu_1 u_1^3$, $f_2(u_2)=\mu_2 u_2^3$, $F(u_1,u_2)=\frac{1}{2}\beta u_1^2 u_2^2$, $N=3$ together with the conditions
\begin{equation}
\int_{\mathbb{R}^3} u_1^2=a_1^2, \int_{\mathbb{R}^3} u_2^2=a_2^2,
\end{equation}
where $a_1$, $a_2$, $\mu_1$, $\mu_2$, $\beta>0$, $\lambda_1$ and $\lambda_2$ are Lagrange multipliers, they proved that there exists a $\beta_1>0$ depending on $a_i$ and $\mu_i$ such that if $0<\beta<\beta_1$ then (1.1)-(1.2) has a solution $(\lambda_1,\lambda_2,\bar u_1,\bar u_2)$ where $\lambda_1$, $\lambda_2<0$ and $\bar u_1$ and $\bar u_2$ are both postive and radially symmetric and there exists $\beta_2>0$ depending on $a_i$ and $\mu_i$ such that, if $\beta>\beta_2$, then (1.1)-(1.2) has a solution $(\lambda_1,\lambda_2,\bar u_1,\bar u_2)$ where $\lambda_1$, $\lambda_2<0$ and $\bar u_1$ and $\bar u_2$ are both positive and radially symmetric, other interesting results for normalized solutions can be found in \cite{B-1}, \cite{B-4}, \cite{B-5}, \cite{B-6}, \cite{B-7}, \cite{G-Z}, \cite{L-Z}, \cite{NTV14}, \cite{NTV15}, \cite{NTV18}, \cite{PV17}, \cite{ZhZh} and references therein.
\par
In \cite{LZ-11} K. Li and Z. T. Zhang studied the case when $\lambda_1$ and $\lambda_2$ are fixed and $f_1(u_1)=\mu_1 u_1^3$, $f_2(u_2)=\mu_2 u_2^3$, $F(u_1,u_2)=\frac{1}{2}\beta u_1^2 u_2^2+\kappa u_1u_2$,  where $\lambda_1$, $\lambda_2$, $\mu_1$, $\mu_2>0$ and $\beta$, $\kappa\in\mathbb{R}$, they proved that
when $0<|\kappa|<\sqrt{\lambda_1\lambda_2}$ (1.1) has a solution $(u_1,u_2)$ such that $u_1$, $u_2>0$ if $\kappa>0$, and $u_1>0$, $u_2<0$ or $u_1<0$, $u_2>0$ if $\kappa<0$, other results of Schr\"{o}dinger systems with linear and nonlinear couplings can be found in \cite{LZ-2}, \cite{TZ-1}, \cite{ZZ} etc. These  papers inspire us to consider the Schr\"{o}dinger systems with linear and nonlinear couplings
%More generally we can consider
\begin{equation}\label{yun1-18-1}
\begin{cases}
-\Delta u_1-\lambda_1 u_1=\mu_1 |u_1|^{p_1-2}u_1+r_1\beta |u_1|^{r_1-2}u_1|u_2|^{r_2}+\kappa (x)u_2\ \textup{in}\ \mathbb{R}^N,\\
-\Delta u_2-\lambda_2 u_2=\mu_2 |u_2|^{p_2-2}u_2+r_2\beta |u_1|^{r_1}|u_2|^{r_2-2}u_2+\kappa (x)u_1\ \textup{in}\ \mathbb{R}^N,\\
 u_1\in H^1(\mathbb{R}^N), u_2\in H^1(\mathbb{R}^N),
\end{cases}
\end{equation}
with the condition
\begin{equation}\label{yun1-18-6}
\int_{\mathbb{R}^N} u_1^2=a_1^2, \int_{\mathbb{R}^N} u_2^2=a_2^2.
\end{equation}
Let $H^1(\mathbb{R}^N)$ be the usual Sobolev space and denote  its norm by
$$\|u\|:=\|u\|_{H_r^1}:=(|\nabla u|_2^2+|u|_2^2)^{1/2}.$$
 In order to use the compact embedding in whole space, we denote the radially symmetric subspace as follows
$$H_r^1:=H_{rad}^1(\mathbb {R}^N):=\{u\in H^1(\mathbb{R}^N): u(x)=u(|x|)\}.$$

We set
 $$S_i:=S_{a_i}:=\{u\in H_r^1: |u|_2=a_i\},\ i=1,2,$$
 where $|u|_p:=|u|_{p,\mathbb{R}^N}:=(\int_{\mathbb{R}^N}|u|^p)^{{1}/{p}},\ p>1$.
 From standard variational arguments we know that critical points of the following functional on $S_1\times S_2$ are weak solutions of \eqref{yun1-18-1}-\eqref{yun1-18-6},
\begin{equation}
\begin{split}
J(u_1,u_2)=&\frac{1}{2}(\int_{\mathbb{R}^N} |\nabla u_1|^2+\int_{\mathbb{R}^N} |\nabla u_2|^2)-\frac{\mu_1}{p_1}|u_1|_{p_1}^{p_1}-\frac{\mu_2}{p_2}|u_2|_{p_2}^{p_2}\\
&-\beta\int_{\mathbb{R}^N} |u_1|^{r_1}|u_2|^{r_2}-\int_{\mathbb{R}^N} \kappa(x) u_1u_2.\nonumber
\end{split}
\end{equation}
 We assume $\kappa(x)=\kappa(|x|)$, $\kappa(x)\in L^{p}(\mathbb{R}^N)$ where $N/2<p<\infty$ from the Palais's principle of symmetric criticality, the critical point of $J|_{S_1\times S_2}$ on $H_r^1\times H_r^1$ is the critical point of $J|_{S_1\times S_2}$ on $H^1(\mathbb{R}^N)\times H^1(\mathbb{R}^N)$. We just need to look for critical points of $J|_{S_1\times S_2}$ on $H^1_r\times H^1_r$.
 By regularity theory of elliptic equations, weak solutions of \eqref{yun1-18-1}-\eqref{yun1-18-6} are classical.
It is easy to see that $S_1\times S_2$ is a $C^2$ Finsler manifold modeled on the Hilbert space $H^1_r\times H^1_r$. We consider the functional $J(u_1,u_2)$ on the manifold $S_1\times S_2$, by Gagliardo-Nirenberg  inequality we have that $J_{S_1\times S_2}$ is bounded from below for \textbf{$L^2$-subcritical case:} $2<p_1<2+4/N$, $2<p_2<2+4/N$ and $2<r_1+r_2<2+4/N$, we will use minimizing method and Ekland
variational principle to get a minimum point of $J_{S_1\times S_2}$. For \textbf{$L^2$-supercritical case:} $2+4/N<p_1<2^{\star}$, $2+4/N<p_2<2^{\star}$ and $2+4/N<r_1+r_2<2^{\star}$, where $2^*=2N/(N-2)$, the functional $J_{S_1\times S_2}$ is not bounded from below, we try to construct a mountain pass structure of $J$ on the manifold $S_1\times S_2$ and by the minimax theory on the Finsler manifold which was introduced in \cite{GH-1} and to obtain the critical point of $J$ on $S_1\times S_2$.
%  From above discussion we study the following two cases:
% \textbf{$L^2-$ subcritical case:} $2<p_1<2+4/N$, $2<p_2<2+4/N$ and $2<r_1+r_2<2+4/N$.
%\textbf{$L^2-$ supercritical case:} $2+4/N<p_1<2^{\star}$, $2+4/N<p_2<2^{\star}$ and $2+4/N<r_1+r_2<2^{\star}.$

\iffalse\textbf{Structure of this paper:} In the section 3 we consider the $L^2-$ subcritical case in $\Omega\subset \mathbb{R}^N$ is bounded smooth domain and $\Omega=\mathbb{R}^N$ by using Ekland variational principle. In the section 4 we consider some cases in the $L^2$ supercritical case in $\Omega\subset \mathbb{R}^N$ is bounded smooth domain and $\Omega=\mathbb{R}^3$ by using mountain pass theorem.\fi

In section 2, we  consider the $L^2$-subcritical case by Ekland variational principle on the manifold $S_1\times S_2$,
%As $\Omega$ is a bounded smooth domain we have following two theorems.
%\begin{thm}
%Assume that $2<p_1<2+4/N$, $2<p_2<2+4/N$, $2<r_1+r_2<2+4/N$, $\Omega\subset \mathbb{R}^{N}$, $N\geq 1$ is a bounded smooth domain and $0\leq\kappa(x)\in L^{\infty}(\Omega)$,  then \eqref{yun1-18-1} and \eqref{yun1-18-6} has a solution $(\lambda_1,\lambda_2,u_{1,0},u_{2,0})$, where $u_{1,0}\geq 0$, $u_{2,0}\geq 0$ and $(u_{1,0},u_{2,0})$ is a minimizer of $J(u_1,u_2)$ on $S_1\times S_2$.
%\end{thm}
%\begin{thm}
%Assume that $2<p_1<2+4/N$, $2<p_2<2+4/N$, $2<r_1+r_2<2+4/N$, $\Omega\subset \mathbb{R}^{N}$, $N\geq 1$ is a bounded smooth domain and $0\geq\kappa(x)\in L^{\infty}(\Omega)$,  then \eqref{yun1-18-1} and \eqref{yun1-18-6} has a solution $(\lambda_1,\lambda_2,u_{1,0},u_{2,0})$ where $u_{1,0}\geq 0$, $u_{2,0}\leq 0$ or $u_{1,0}\leq 0$, $u_{2,0}\geq 0$ and $(u_{1,0},u_{2,0})$ is a minimizer of $J(u_1,u_2)$ on $S_1\times S_2$.
%\end{thm}
by  Liouville types theorems in \cite{IK-1}, we obtain the following three theorems.
\begin{thm}\label{th3-23-1}
Assume $\beta>0$, $\kappa(x)>0$, $\kappa(x)=\kappa(|x|)$ and $\kappa(x)\in L^{{p}}(\mathbb{R}^{N})\cap L^{{\infty}}(\mathbb{R}^{N})$, $N/2<p<\infty$.
If one of the following assumptions is satisfied:
\par
(a) $2\leq N\leq 4$ and $2<p_1$, $p_2$, $r_1+r_2<2+4/N$, or
\par
(b) $N\geq 5$ and $2<p_1$, $p_2<2+2/(N-2)$ and $r_1+r_2<2+4/N$.

Then \eqref{yun1-18-1}-\eqref{yun1-18-6} has a solution $(\lambda_{1,0},\lambda_{2,0},u_{1,0},u_{2,0})$ such that  $\lambda_{1,0}$, $\lambda_{2,0}<0$ and $u_{1,0}$, $u_{2,0}>0$.
\noindent Moreover, $ u_{1,0}$ and $ u_{2,0}$ are radially symmetric.
\end{thm}
Similarly, when $\kappa(x)<0$ we obtain
\begin{thm}\label{th3-23-2}
Assume $\beta>0$, $\kappa(x)<0$, $\kappa(x)=\kappa(|x|)$ and $\kappa(x)\in L^{{p}}(\mathbb{R}^{N})\cap L^{{\infty}}(\mathbb{R}^{N})$, $N/2<p<\infty$. If one of the following assumptions is satisfied:

(a) $2\leq N\leq 4$ and $2<p_1$, $p_2$, $r_1+r_2<2+4/N$, or

(b) $N\geq 5$ and $2<p_1$, $p_2<2+2/(N-2)$ and $2<r_1+r_2<2+4/N$.
\par
Then \eqref{yun1-18-1}-\eqref{yun1-18-6} has a solution $(\lambda_{1,0},\lambda_{2,0},u_{1,0},u_{2,0})$ such that  $\lambda_{1,0}$, $\lambda_{2,0}<0$ and $u_{1,0}>0$, $u_{2,0}<0$ or $u_{1,0}<0$, $u_{2,0}>0$. Moreover, $ u_{1,0}$ and $ u_{2,0}$ are radially symmetric.
\end{thm}

In section 3, we consider \eqref{yun1-18-1}-\eqref{yun1-18-6} when $N=3$, $p_1=p_2=4$, and $r_1=r_2=2$ which is $L^2-$ supercritical case.
\begin{equation}\label{yun1-18-5}
\begin{cases}
-\Delta u_1-\lambda_1 u_1=\mu_1 u_1^3+\beta u_1u_2^2+\kappa (x)u_2\ \textup {in}\ \mathbb{R}^3,\\
-\Delta u_2-\lambda_2 u_2=\mu_2 u_2^3+\beta u_1^2 u_2+\kappa (x)u_1\ \textup {in}\ \mathbb{R}^3,\\  u_1\in H^1(\mathbb{R}^3),u_2\in H^1(\mathbb{R}^3),
\end{cases}
\end{equation}
with the condition
\begin{equation}\label{yun1-18-2}
\int_{\mathbb{R}^3} u_1^2=a_1^2, \int_{\mathbb{R}^3} u_2^2=a_2^2,
\end{equation}
where $a_1$, $a_2$, $\mu_1$, $\mu_2>0$, $\beta\in \mathbb{R}$ and $\kappa(x)\in L^\infty(\mathbb{R}^3)$, \eqref{yun1-18-5}-$\eqref{yun1-18-2}$ is the classical Bose-Einstein-condensates model.

 By constructing the mountain pass structure on manifold $S_1\times S_2$, we have
\begin{thm}\label{zh3-23}

Assume $\beta>0$, $\kappa(x)>0$, $\kappa(x)=\kappa(|x|)$, $\kappa(x)\in L^p(\mathbb{R}^3)\cap L^{\infty}(\mathbb{R}^3)$ for some $\frac{3}{2}<p<\infty$, $\frac{2}{3}\nabla \kappa(x)\cdot x+\kappa(x)\geq 0$, $\nabla \kappa(x)\cdot x$ is bounded and
$$
|\kappa(x)|_{\infty}<\frac{5}{18C_{a_1,a_2}^2a_1a_2},
$$
 where $C_{a_1,a_2}=((\mu_1+\beta)a_1S^4+(\mu_2+\beta)a_2S^4)$ and $S$ denotes the  Sobolev embedding constant in $\mathbb{R}^3$. Then \eqref{yun1-18-5}-\eqref{yun1-18-2} has a solution $(\bar\lambda_1,\bar\lambda_2,\bar u_1,\bar u_2)$ such that $\bar\lambda_1<0$, $\bar\lambda_2<0$, $\bar u_1>0$, $\bar u_2>0$. Moreover, $\bar u_1$ and $\bar u_2$ are radially symmetric.
\end{thm}
\begin{remark}
From Theorem \ref{zh3-23} we have that when $a_1$ or $a_2$ is small, then the upper bound of $\kappa(x)$ can be large. %For example we can take $\kappa(x)=O(1/{|x|^p})$ as $|x|\rightarrow \infty$ where $0<p\leq 3/2$.
\end{remark}
\begin{remark}
For example we can take $\kappa(x)=\frac{c}{1+|x|^{\frac{3}{2}}}$ where $c<\frac{5}{18C_{a_1,a_2}^2a_1a_2}$, which satisfies all conditions in Theorem \ref{zh3-23}.
\end{remark}
At the end of section 3, we will relax some restrictions of $\kappa(x)$ to obtain the same conclusion as Theorem \ref{zh3-23}.

\section{$L^2$-subcritical case}
 In this section, we prove Theorem \ref{th3-23-1} and Theorem \ref{th3-23-2}.\par
 We assume $\kappa(x)$ satisfies $\kappa(x)\in L^{\infty}(\mathbb{R}^N)$, $\kappa(x)> 0$ or $\kappa(x)< 0$, $2<p_1<2+{4}/{N}$, $2<p_2<2+{4}/{N}$, $2<r_1+r_2<2+{4}/{N}$, $\beta>0$, we work on the space $H^1_{r}\times H^1_{r}$,
  the corresponding energy functional of \eqref{yun1-18-1}-\eqref{yun1-18-6} on $S_1\times S_2$ is
\begin{equation}
\begin{split}
J(u_1,u_2)=&\frac{1}{2}(\int_{\mathbb{R}^N} |\nabla u_1|^2+\int_{\mathbb{R}^N} |\nabla u_2|^2)-\frac{\mu_1}{p_1}|u_1|_{p_1}^{p_1}-\frac{\mu_2}{p_2}|u_2|_{p_2}^{p_2}\\
&-\beta\int_{\mathbb{R}^N} |u_1|^{r_1}|u_2|^{r_2}-\int_{\mathbb{R}^N} \kappa(x) u_1u_2.\nonumber
\end{split}
\end{equation}
We try to find the critical point of $J$ on $S_1\times S_2$.
\begin{lemma}[Gagliardo-Nirenberg inequality]\label{3-23-1}
For any $u\in H^1(\mathbb{R}^N)$ we have
$$
|u|_p\leq C_{N,p} |\nabla u|^{\alpha}_2|u|^{1-\alpha}_2,
$$
where $\alpha=\frac{N(p-2)}{2p}$.
\end{lemma}
\begin{lemma}\label{3-24-1}
$J(u_1,u_2)$ is coercive and bounded from below on $S_1\times S_2$.
\end{lemma}

\begin{proof}
By lemma 2.1 we have
$$\int_{\mathbb{R}^N} |u_i|^{p_i}\leq C(N,p_i,a_i)|\nabla u_i|_2^{\frac{N(p_i-2)}{2}}\ , \ i=1,2.$$
Set $C_1= C(N,p_1,a_1)$, $C_2= C(N,p_2,a_2)$ and we can find $q,q'$ such that $2\leq r_1q$, $r_2q'\leq 2^{*}$, $1/q+1/q'=1$, when $2\leq r_1+r_2\leq 2^*$ we obtain
$$\int_{\mathbb{R}^N} |u_1|^{r_1} |u_2|^{r_2}\leq |u_1|_{r_1q}^{r_1}|u_2|_{r_2q'}^{r_2}\leq C|\nabla u_1|_2^{\frac{N(r_1q-2)}{2q}}|\nabla u_2|_2^{\frac{N(r_2q'-2)}{2q'}},$$
where $C_3=C_3(N,a_i,r_i,q,q')$. By direct computation, when $r_1+r_2<2+4/N$ we get that
$$\frac{N(r_1q-2)}{2q}+\frac{N(r_2q'-2)}{2q'}<2.$$
By Young inequality we can find $\gamma_1,\gamma_2<2$ such that
$$\int_{\mathbb{R}^N} |u_1|^{r_1} |u_2|^{r_2}\leq C_3(|\nabla u_1|^{\gamma_1}+|\nabla u_2|^{\gamma_2}).$$
By H\"{o}lder inequality, we get
$$\int_{\mathbb{R}^N} \kappa(x) u_1 u_2\leq |\kappa(x)|_{\infty} a_1a_2=:C_4.$$
Thus we have for any $(u_1,u_2)\in S_1\times S_2,$
\begin{equation}
\begin{split}
J(u_1,u_2)\geq& \frac{1}{2}(|\nabla u_1|_2^2+|\nabla u_2|_2^2)-\frac{C_1\mu_1}{p_1}|\nabla u_1|_2^{\frac{N(p_1-2)}{2}}\\
&-\frac{C_2\mu_2}{p_2}|\nabla u_2|_2^{\frac{N(p_2-2)}{2}}-C_3\beta(|\nabla u_1|_2^{\gamma_1}+|\nabla u_2|_2^{\gamma_2})-C_4,\nonumber
\end{split}
\end{equation} so $J(u_1,u_2)$ is coercive and bounded from below on $S_1\times S_2$.
\end{proof}

We need some Liouville type theorems to ensure the weak limit of a PS sequence is not zero.  If we assume $\kappa(x)\geq 0$ we can choose the minimizing sequence of $J$ on $S_1\times S_2$ is nonnegative.
%The following Liouville type theorem can be found in \cite{IK-1}.
\begin{lemma}[See \cite{IK-1}]\label{3-24-2}Assume that $u$ is a smooth function in $\mathbb{R}^N$,

(a) Suppose that $q\in (1,{N}/{(N-2)}]$ when $N\geq 3$ and $q\in (1,\infty)$ when $N=1,2$. Let $u\in L^q(\mathbb{R}^N)$ be a smooth nonnegative function satisfying $-\Delta u\geq 0$ in $\mathbb{R}^N$, then $u\equiv 0$.

(b) Suppose that $q\in (1, 1+{2}/{(N-2)}]$ the inequality $-\Delta u\geq u^q$ does not have a positive classical solution in $\mathbb{R}^N$.
\end{lemma}

\begin{lemma}\label{le3-25-1}
 Assume $\kappa(x)\geq 0$ and $(\lambda_1,\lambda_2,u_1,u_2)\in \mathbb{R}^2\times H_r^1\times H_r^1$ is a solution of \eqref{yun1-18-1}, and assume $p_1,\ p_2<2+4/N$ when $N\leq 4$ ,  $p_1,\ p_2<2+N/(N-2)$ when $N\geq 5$, we have if $u_1\geq 0,u_1\not\equiv 0,u_2\geq 0$, then $\lambda_1<0$, if $u_2\geq 0$, $u_2\not\equiv 0$, $u_1\geq 0$, then $\lambda_2<0$.
\end{lemma}
\begin{proof}
In the first case when $u_1\not\equiv 0$, note that $\kappa (x)\geq 0$ and if $\lambda_1\geq 0$ we have
$$-\Delta u_1=\lambda_1u_1+\mu_1 u_1^{p_1-1}+r_1\beta u_1^{r_1-1}u_2^{r_2}+\kappa(x)u_2\geq 0,$$
then by Lemma \ref{3-24-2} we can deduce that $u_1\equiv 0$ which is impossible, then $\lambda_1<0$. Similarly we can deduce the rest part of this lemma.
\end{proof}
 By Lemma \ref{3-24-1} we have $J(u_1,u_2)$ is bounded from below and coercive on $S_1\times S_2$ then we can find a minimizing sequence $\{(v_{1,n},v_{2,n})\}$ for $J|_{S_1\times S_2}$ and because $\kappa(x)\geq 0$ we can assume $v_{i,n}\geq 0$, $i=1,2$. By Ekland variational principle we have $\{(u_{1,n},u_{2,n})\}$ is a PS sequence for $J|_{S_1\times S_2}$ at level $c$, where
$$c:=\inf_{S_1\times S_2}J(u_1,u_2),$$
and $\|v_{i,n}-u_{i,n}\|_{H^1_r}\rightarrow 0$ as $n\rightarrow \infty$. Moreover, $\{(u_{1,n},u_{2,n})\}$ is bounded, then we have
\begin{equation}\label{eq3-25-1}
(u_{1,n},u_{2,n})\rightharpoonup (u_{1,0},u_{2,0})\in H^1_r\times H^1_r,
\end{equation}
by standard arguments of compact embedding, we have $u_{i,0}\geq 0$, $i=1,2$.
\par
From above discussion we have
\begin{equation}\label{3-3}
J|'_{S_1\times S_2}(u_{1,n},u_{2,n})=J'(u_{1,n},u_{2,n})-\lambda_{1,n}(u_{1,n},0)-\lambda_{2,n}(0,u_{2,n})\rightarrow 0
\end{equation}
 in $(H_r^1\times H_r^1)^*$, where
\begin{equation}
\begin{split}
\lambda_{1,n}&=\frac{1}{|u_{1,n}|_2^2}(J'(u_{1,n},u_{2,n}),(u_{1,n},0))\\
&=\frac{1}{a_1^2}(\int_{\mathbb{R}^N} |\nabla u_{1,n}|^2-\mu_1\int_{\mathbb{R}^N} |u_{1,n}|^{p_1}-\beta \int_{\mathbb{R}^N} |u_{1,n}|^{r_1}|u_{2,n}|^{r_2}-\int_{\mathbb{R}^N} \kappa(x) u_{1,n}u_{2,n}),\nonumber
\end{split}
\end{equation}
\begin{equation}
\begin{split}
\lambda_{2,n}&=\frac{1}{|u_{2,n}|_2^2}(J'(u_{1,n},u_{2,n}),(0,u_{2,n}))\\
&=\frac{1}{a_2^2}(\int_{\mathbb{R}^N} |\nabla u_{2,n}|^2-\mu_2\int_{\mathbb{R}^N} |u_{2,n}|^{p_2}-\beta \int_{\mathbb{R}^N} |u_{1,n}|^{r_1}|u_{2,n}|^{r_2}-\int_{\mathbb{R}^N} \kappa(x) u_{1,n}u_{2,n}).\nonumber
\end{split}
\end{equation}
are bounded sequences and we may assume $\lambda_{1,n}\rightarrow \lambda_{1,0}$, $\lambda_{2,n}\rightarrow \lambda_{2,0}$ up to the subsequence. Then by weak convergence we have
\begin{equation}
J'(u_{1,0},u_{2,0})-\lambda_{1,0}(u_{1,0},0)-\lambda_{2,0}(0,u_{2,0})=0\ \textup{in}\ (H_r^1\times H_r^1)^*.\nonumber
\end{equation}
Moreover, $(\lambda_1,\lambda_2,u_{1,0},u_{2,0})$ is a solution of \eqref{yun1-18-1}.
In order to obtain $(u_{1,0},u_{2,0})$ also satisfies \eqref{yun1-18-6}, we need the following lemma.
\begin{lemma}\label{le3-25-2}
 If $\kappa(x)\geq 0$, $\kappa (x)\in L^p(\mathbb{R}^N)$ for some ${N}/{2}<p<\infty$ and $\lambda_{i,0}<0$ then $u_{i,n}\rightarrow u_{i,0}$, $i=1,2$. As a consequence if $\lambda_{i,0}<0$, $i=1,2$, then $(\lambda_{1,0},\lambda_{2,0},u_{1,0},u_{2,0})$ is a solution of \eqref{yun1-18-1}-\eqref{yun1-18-6}.
\end{lemma}
\begin{proof}

Because $\kappa (x)\in L^{p}(\mathbb{R}^N)$ where $p>{N}/{2}$, then by the compact embedding and H\"{o}lder inequality we have
$$\int_{\mathbb{R}^N}\kappa(x)u_{1,n}u_{2,n}\rightarrow \int_{\mathbb{R}^N}\kappa(x)u_{1,0}u_{2,0},$$
$$|u_{i,n}|_{p_i}\rightarrow |u_{i,0}|_{p_i}\ i=1,2,$$
$$\int_{\mathbb{R}^N}|u_{1,n}|^{r_1}|u_{2,n}|^{r_2}\rightarrow \int_{\mathbb{R}^N}|u_{1,0}|^{r_1}|u_{2,0}|^{r_2},$$
as $n\rightarrow\infty$. Notice that $\lambda_1<0$, by \eqref{3-3} and weak convergence we can deduce that
$$
(J'(u_{1,n},u_{2,n})-\lambda_{1,0}(u_{1,n},0),(u_{1,n},0))\rightarrow 0,
$$
$$
(J'(u_{1,0},u_{2,0})-\lambda_{1,0}(u_{1,0},0),(u_{1,0},0))= 0.
$$
then from the above five equations we have that
$$
|\nabla u_{1,n}|_2^2-\lambda_{1,0}|u_{1,n}|_2^2\rightarrow |\nabla u_{1,0}|_2^2-\lambda_{1,0}|u_{1,0}|_2^2.
$$
Because $\lambda_1<0$, we get $u_{1,n}\rightarrow u_{1,0}$ in  $H^1_r $. Similarly if $\lambda_2<0$ we can get $u_{2,n}\rightarrow u_{2,0}$ in $H^1_r$. As a consequence when $\lambda_{1,0}$ and $\lambda_{2,0}$ are both negative, then $(\lambda_{1,0},\lambda_{2,0},u_{1,0},u_{2,0})$ is a solution of \eqref{yun1-18-1}-\eqref{yun1-18-6}. Moreover, by maximum principle we have $u_{1,0}$ and $u_{2,0}$ are positive.
\end{proof}

Next we consider  the single equation which is useful in the following proof,
\begin{equation}\label{yun1-20-1}
-\Delta u+\lambda u=\mu |u|^{p-2}u\ \textup{in}\ \mathbb{R}^N,
\end{equation}
with the condition
$$\int_{{\mathbb R}^N}u^2=a^2,$$
where $\mu>0$, $2<p<2+4/N$ and $\lambda$ is Lagrangian multiplier. By Lemma 3.1 of \cite{B-1}, we know the corresponding energy functional of \eqref{yun1-20-1} on $S_a$ denotes by
$$I(u):=I_{a,\mu}(u)=\frac{1}{2}\int_{\mathbb{R}^N} |\nabla u|^2-\frac{\mu}{p}\int_{\mathbb{R}^N} |u|^p,$$
and the least energy of $I$ on $S_a$ is denoted by
$$m_p^{\mu}(a)=\inf_{S_a}I(u),$$
which is achieved at some $u_a\in H^1_r$, and there exists $\lambda_a>0$ such that $(\lambda_a,u_a)$ is a solution of \eqref{yun1-20-1} and $|u_a|_2=a$. Moreover, $m_p^{\mu}(a)<0$.
\par
\begin{proof}[Proof of Theorem 1.1]
%\noindent\textbf{Proof of theorem 1.1.}
 We denote the energy functional $\bar J$ for \eqref{yun1-18-1}-\eqref{yun1-18-6} on $S_1\times S_2$ when $\kappa(x)=0$, where % which introduced in \cite{B-1}
$$
\bar J(u_1,u_2)=\frac{1}{2}(\int_{\mathbb{R}^N} |\nabla u_1|^2+\int_{\mathbb{R}^N} |\nabla u_2|^2)-\frac{\mu_1}{p_1}|u_1|_{p_1}^{p_1}-\frac{\mu_2}{p_2}|u_2|_{p_2}^{p_2}-\beta\int_{\mathbb{R}^N} |u_1|^{r_1}|u_2|^{r_2}.
$$
 Since $\beta>0$ and $\kappa(x)\geq 0$
we have $J(u_1,u_2)\leq \bar J(u_1,u_2)$, if $u_1\geq 0$, $u_2\geq 0$. Notice that $\bar J(u_1,u_2)=\bar J(|u_1|,|u_2|)$, we get
\begin{equation}
\begin{split}
c:=\inf_{S_1\times S_2} J(u_1,u_2)&\leq\inf_{\substack{S_1\times S_2 \\u_1,u_2\geq 0}}J(u_1,u_2)\\
&\leq \inf_{\substack{S_1\times S_2\\u_1,u_2\geq 0}}\bar J(u_1,u_2)\\
&=\inf_{S_1\times S_2}\bar J(u_1,u_2)\\
&\leq m_{p_1}^{\mu_1}(a_1)+m_{p_2}^{\mu_2}(a_2)\\
&<0.\nonumber
\end{split}
\end{equation}

Then for the weak limit $(u_{1,0},u_{2,0})$ of the PS sequence $\{(u_{1,n},u_{2,n})\}$ which is obtained in \eqref{eq3-25-1}, we have the following four cases.

(i) If $(u_{1,0},u_{2,0})=(0,0)$, then by compact embedding we have
$$
0>c=\lim_{n\rightarrow \infty}J(u_{1,n},u_{2,n})\geq \liminf_{n\rightarrow \infty}\frac{1}{2}(|\nabla u_{1,n}|_2^2+|\nabla u_{2,n}|_2^2)\geq 0,
$$
which is impossible.

(ii) If $u_{1,0}\not\equiv 0$, but $u_{2,0}\equiv 0$, because $(\lambda_1,\lambda_2,u_{1,0},u_{2,0})$ is a solution of \eqref{yun1-18-1} then we must have $u_{1,0}\equiv 0$, because $\kappa(x)>0$, which is impossible.

(iii) If $u_{1,0}\equiv 0$ but $u_{2,0}\not\equiv 0$, it is same as (ii), which is impossible.

(iv) If $u_{1,0}\not\equiv 0$ and $u_{2,0}\not\equiv 0$, by Lemma \ref{le3-25-1} we have $\lambda_{1,0},\lambda_{2,0}<0$, by  Lemma \ref{le3-25-2}, we have $u_{i,n}\rightarrow u_{i,0}$ in $H^1_r$, by maximum principle we have $u_{1,0}$, $u_{2,0}>0$, then we finish the proof.\end{proof}

Now we consider the case $\kappa(x)\leq 0$.\par
 Because $J$ is bounded from below and coercive, we can find a minimizing sequence $(v_{1,n},v_{2,n})\in S_1\times S_2$. Notice that $\kappa(x)\leq 0$ then without loss of generality we may assume $v_{1,n}\leq 0$ and $v_{2,n}\geq 0$, then by Ekland variational principle we have that there exists a PS sequence for $J|_{S_1\times S_2}$ at level $c$, where
$$
c:=\inf_{S_1\times S_2}J(u_1,u_2),
$$
and $\|v_{i,n}-u_{i,n}\|_{H_r^1}\rightarrow 0$, $i=1,2$. By the coerciveness of $J$ on $S_1\times S_2$, there exists $(u_{1,0},u_{2,0})\in H_r^1\times H_r^1$ such that
\begin{equation}\label{eq3-25-2}
(u_{1,n},u_{2,n})\rightharpoonup(u_{1,0},u_{2,0}) \ \ \textup{in}\ H_r^1\times H_r^1,
\end{equation}
in additional $u_{1,0}\leq 0$, $u_{2,0}\geq 0$. Similarly to prove Theorem \ref{th3-23-1} we can get $\lambda_{1,0}, \lambda_{2,0}\in \mathbb{R}$ such that $(\lambda_{1,0},\lambda_{2,0},u_{1,0},u_{2,0})$ is a solution of \eqref{yun1-18-1}. We also have following lemma, which is similar to Lemma \ref{le3-25-1}.
\begin{lemma}\label{le3-25-3}
Assume $\kappa(x)\leq 0$ and $(\lambda_1,\lambda_2,u_1,u_2)\in \mathbb{R}^2\times H_r^1\times H_r^1$ is a solution of \eqref{yun1-18-1}, and assume $p_1,\ p_2<2+4/N$ when $N\leq 4$,  $p_1,\ p_2<2+N/(N-2)$ when $N\geq 5$, we have if $u_1\leq 0,u_1\not\equiv 0,u_2\geq 0$, then $\lambda_1<0$; if $u_2\geq 0$, $u_2\not\equiv 0$, $u_1\leq 0$, then $\lambda_2<0$.
\end{lemma}
\begin{proof}
When $N\leq4$, for the first case, we prove it by contradiction. If $\lambda_1\geq 0$, we have $-u_1\geq 0$, $-u_1\not\equiv 0$
$$
-\Delta (-u_1)=\lambda_1(-u_1)+\mu_1|u_1|^{p_1-2}(-u_1)+\beta r_1|u_1|^{r_1-2}(-u_1)u_2^{r_2}+(-\kappa(x) u_2).
$$
So we get $-\Delta (-u_1)\geq 0$, by Lemma \ref{3-24-2} we know $u_1\equiv 0$, which is impossible, thus $\lambda_1<0$. Similarly if $u_2\geq 0$, $u_2\not\equiv 0$ and $u_1\leq 0$, then $\lambda_2<0$.\par
When $N\geq 5$ note that $-\Delta (-u_1)\geq \mu_1(-u_1)^{p_1-1}$, $u_1\leq 0$, and by Lemma \ref{3-24-2} we have $u_1\equiv 0$, which is impossible. The  rest part of the proof  is the same as Lemma \ref{le3-25-1}.
\end{proof}
In order to obtain that $(u_{1,0},u_{2,0})$ in \eqref{eq3-25-2} satisfies $(1.4)$, we need the following lemma.
\begin{lemma}\label{le3-25-4}
If $\kappa(x)\leq 0$, $\kappa (x)\in L^p(\mathbb{R}^N)$ for some $N/2<p<\infty$ and $\lambda_{i,0}<0$, then $u_{i,n}\rightarrow u_{i,0}$, $i=1,2$. As a consequence if $\lambda_{i,0}<0$, $i=1,2$, then $(\lambda_{1,0},\lambda_{2,0},u_{1,0},u_{2,0})$ is a solution of \eqref{yun1-18-1}-\eqref{yun1-18-6}.
\end{lemma}
\begin{proof}
Same as Lemma 2.5.
\end{proof}
Now we prove the existence of solution for \eqref{yun1-18-1}-\eqref{yun1-18-6} when $\kappa(x)<0$.

\begin{proof}[Proof of theorem 1.2.] Using method of the proof of Theorem \ref{th3-23-1} and noticing that $\kappa(x)\leq 0$, we have
\begin{equation}
\begin{split}
c:=\inf_{S_1\times S_2} J(u_1,u_2)&\leq\inf_{\substack{S_1\times S_2 \\u_1\leq 0,u_2\geq 0}}J(u_1,u_2)\\
&\leq \inf_{\substack{S_1\times S_2\\u_1\leq 0,u_2\geq 0}}\bar J(u_1,u_2)\\
&=\inf_{S_1\times S_2}\bar J(u_1,u_2)\\
&\leq m_{p_1}^{\mu_1}(a_1)+m_{p_2}^{\mu_2}(a_2)\\
&<0.\nonumber
\end{split}
\end{equation}
Then we have following four cases.

(i) If $(u_{1,0},u_{2,0})=(0,0)$, then by compact embedding we have
$$
0>c=\lim_{n\rightarrow \infty}J(u_{1,n},u_{2,n})\geq \liminf_{n\rightarrow \infty}\frac{1}{2}(|\nabla u_{1,n}|_2^2+|\nabla u_{2,n}|_2^2)\geq 0,
$$
which is impossible.

(ii) If $u_{1,0}\not\equiv 0$ but $u_{2,0}\equiv 0$, because $(\lambda_1,\lambda_2,u_{1,0},u_{2,0})$ is a solution of \eqref{yun1-18-1}, then we must have $u_{1,0}\equiv 0$ because $\kappa(x)<0$, which is impossible.

(iii) If $u_{1,0}\equiv 0$ but $u_{2,0}\not\equiv 0$, the proof is same as (ii), it is impossible.

(iv) If $u_{1,0}\not\equiv 0$ and $u_{2,0}\not\equiv 0$, by Lemma \ref{le3-25-3}, we have $\lambda_{1,0}$, $\lambda_{2,0}<0$. Then by Lemma \ref{le3-25-4} $u_{i,n}\rightarrow u_{i,0}$ in $H^1_r$, by maximum principle we have $u_{1,0}<0$, $u_{2,0}>0$. Finally if we take $\{(v_{1,n},v_{2,n})\}$ as the minimizing sequence of $J|_{S_1\times S_2}$, such that $v_{1,n}\geq 0$, $v_{2,n}\leq 0$, then we can get $u_{1,0}>0$, $u_{2,0}<0$, and there exists $(\lambda_{1,0},\lambda_{2,0})$ such that $(\lambda_{1,0},\lambda_{2,0},u_{1,0},u_{2,0})$ is a solution of \eqref{yun1-18-1}-\eqref{yun1-18-6}, then we finish the proof.\end{proof}

\begin{remark}
Moreover, by the properties of  Schwarz rearrangement and $\beta>0$, we can easily deduce that the solution which is found in Theorem \ref{th3-23-1} and Theorem \ref{th3-23-2} is a ground state solution of \eqref{yun1-18-1}-\eqref{yun1-18-6} in $H^1(\mathbb{R}^N)\times H^1(\mathbb{R}^N)$.
\end{remark}
\section{$L^2-$supercritical case}

%\subsection{$L^2$ supercritical and $\Omega=\mathbb{R}^3$}
In this section we consider the solutions for system \eqref{yun1-18-5}-\eqref{yun1-18-2} when $N=3$, $p_1=p_2=4$ and $r_1=r_2=2$, which is $L^2-$supercritical case,
\iffalse \begin{equation}\label{eq:diricichlet}
\begin{cases}
-\Delta u_1-\lambda_1 u_1=\mu_1 u_1^3+\beta u_1u_2^2+\kappa (x)u_2\ \textup {in}\ \mathbb{R}^3,\\
-\Delta u_2-\lambda_2 u_2=\mu_2 u_2^3+\beta u_1^2 u_2+\kappa (x)u_1\ \textup {in}\ \mathbb{R}^3,\\
\int_{\mathbb{R}^3} u_1^2=a_1^2, \int_{\mathbb{R}^3} u_2^2=a_2^2, u_1\in H^1_r.u_2\in H^1_r,
\end{cases}
\end{equation}\fi
the corresponding energy functional on $S_1\times S_2$ is defined by
$$J(u_1,u_2)=\frac{1}{2}\int_{\mathbb{R}^3} |\nabla u_1|^2+|\nabla u_2|^2-\frac{1}{4}\int_{\mathbb{R}^3}\mu_1 u_1^4+\mu_2 u_2^4+2\beta u_1^2u_2^2-\int_{\mathbb{R}^3} \kappa(x)u_1u_2,$$
where $\mu_1$, $\mu_2$, $\beta>0$, $\kappa(x)>0$ and $\kappa(x)\in L^{\infty}(\mathbb{R}^3)$. $J|_{S_1\times S_2}$ is unbounded from below, so we can not achieve $\inf_{S_1\times S_2}J(u_1,u_2)$. In order to get the crtical point of $J|_{S_1\times S_2}$, we will try to find a minimax value of $J|_{S_1\times S_2}$, by constructing a mountain pass structure on $S_1\times S_2$. For this purpose we introduce the following two sets
$$A_{K_1}:=\{(u_1,u_2)\in S_1\times S_2:\int_{\mathbb{R}^3} |\nabla u_1|^2+|\nabla u_2|^2\leq K_1\},$$
$$B_{K_2}:=\{(u_1,u_2)\in S_1\times S_2:\int_{\mathbb{R}^3} |\nabla u_1|^2+|\nabla u_2|^2=K_2\}.$$
By Lemma \ref{3-23-1} we have
$$\int_{\mathbb{R}^3}\mu_1 u_1^4+\mu_2 u_2^4+2\beta u_1^2u_2^2\leq C_{a_1,a_2}(\int_{\mathbb{R}^3} |\nabla u_1|^2+|\nabla u_2|^2)^{\frac{3}{2}},$$
 where $C_{a_1,a_2}=((\mu_1+\beta)a_1S^4+(\mu_2+\beta)a_2S^4)$ and $S>0$ denotes the Sobolev embedding constant in $\mathbb{R}^3$.
\begin{lemma}
There exists $C_1>0$, where $C_1:=C_1(|\kappa(x)|_{\infty},a_1,a_2)$ and $K_1>0$ such that for any $(u_1,u_2)\in A_{K_1}$
\begin{equation}
J(u_1,u_2)>-C_1.
\end{equation}\label{yun1-19-1}
\end{lemma}
\begin{proof}
We let  $K_1<\frac{4}{C_{a_1,a_2}^2}$, where $\frac{4}{C_{a_1,a_2}^2}$ is the biggest zero point of the function
$$\frac{1}{2}x-\frac{C_{a_1,a_2}}{4}x^{\frac{3}{2}}.$$
 Then we have
\begin{equation}
\begin{split}
J(u_1,u_2)&=\frac{1}{2}\int_{\mathbb{R}^3} |\nabla u_1|^2+|\nabla u_2|^2-\frac{1}{4}\int_{\mathbb{R}^3}\mu_1 u_1^4+\mu_2 u_2^4+\beta u_1^2u_2^2-\int _{\mathbb{R}^3}\kappa(x)u_1u_2\\
&\geq \frac{1}{2}\int_{\mathbb{R}^3} |\nabla u_1|^2+|\nabla u_2|^2-\frac{C_{a_1,a_2}}{4}(\int_{\mathbb{R}^3} |\nabla u_1|^2+|\nabla u_2|^2)^{\frac{3}{2}}-|\kappa(x)|_{\infty}a_1a_2\\
&\geq -|\kappa(x)|_{\infty}a_1a_2,\nonumber
\end{split}
\end{equation}
then we take $C_1=|\kappa(x)|_{\infty}a_1a_2$ to get \eqref{yun1-19-1}.
\end{proof}
\begin{lemma}\label{le3-25-7}
Assume $K_2=\frac{16}{9C_{a_1,a_2}^2}$ and $|\kappa(x)|_{\infty}<\frac{5}{18C_{a_1,a_2}^2}$, if $K_1$ small enough, then we have
\begin{equation}\label{yun1-19-2}
\sup_{A_{K_1}}J(u_1,u_2)<\inf_{B_{K_2}}J(u_1,u_2),
\end{equation}
and
\begin{equation}\label{yun1-19-3}
\inf_{B_{K_2}}J(u_1,u_2)>0.
\end{equation}
\end{lemma}
\begin{proof}
Take $(v_1,v_2)\in B_{K_2}$, $(u_1,u_2)\in A_{K_1}$, notice that $K_2>0$ is the maximum point of the function
$$\frac{1}{2}x-\frac{C_{a_1,a_2}}{4}x^{\frac{3}{2}},$$ $|\kappa(x)|_{\infty}<\frac{5}{18C_{a_1,a_2}^2}$, and choose $K_1$ small enough, we have
\begin{equation}\label{eq3-25-3}
\begin{split}
J(v_1,v_2)-J(u_1,u_2)=&\frac{1}{2}\int_{\mathbb{R}^3} |\nabla v_1|^2+|\nabla v_2|^2-\frac{1}{2}\int_{\mathbb{R}^3} |\nabla u_1|^2+|\nabla u_2|^2\\
&-\frac{1}{4}\int_{\mathbb{R}^3} \mu_1v_1^4+\mu_2 v_2^4+2\beta v_1^2 v_2^2+\frac{1}{4}\int_{\mathbb{R}^3} \mu_1 u_1^4+\mu_2 u_2^4+2\beta u_1^2u_2^2\\
&-\int_{\mathbb{R}^3} \kappa(x)v_1v_2+\int_{\mathbb{R}^3} \kappa(x)u_1u_2\\
\geq& \frac{1}{2}\int_{\mathbb{R}^3} |\nabla v_1|^2+|\nabla v_2|^2-\frac{1}{2}\int_{\mathbb{R}^3} |\nabla u_1|^2+|\nabla u_2|^2\\
&-\frac{1}{4}\int_{\mathbb{R}^3} \mu_1v_1^4+\mu_2 v_2^4+2\beta v_1^2 v_2^2-2|\kappa(x)|_{\infty}a_1a_2\\
\geq& \frac{1}{2} K_2-\frac{C_{a_1,a_2}}{4}(K_2)^{\frac{3}{2}}-\frac{1}{2}K_1-2|\kappa(x)|_{\infty}a_1a_2\\
>&0.
\end{split}
\end{equation}
Take $(u_1,u_2)\in B_{K_2}$, similarly to \eqref{eq3-25-3}, we have
\begin{equation}
\begin{split}
J(u_1,u_2)&=\frac{1}{2}\int_{\mathbb{R}^3} |\nabla u_1|^2+|\nabla u_2|^2-\frac{1}{4}\int \mu_1u_1^4+\mu_2 u_2^4+2\beta u_1^2u_2^2-\int_{\mathbb{R}^3} \kappa(x)u_1u_2\\
&\geq \frac{1}{2}K_2-\frac{C_{a_1,a_2}}{4}(K_2)^{\frac{3}{2}}-|\kappa(x)|_{\infty}a_1a_2\\
&>0.\nonumber
\end{split}
\end{equation}
This finishes the proof.
\end{proof}
\iffalse
\begin{lemma}
There exist $K_3$ and $C_3:=C_3(a_i,\mu_i,\beta)$ such that if $|\kappa(x)|_{\infty}<C_3$ then
\begin{equation}\label{yun1-19-3}
\inf_{B_{K_2}}J(u_1,u_2)>0.
\end{equation}
\begin{proof}
take $(u_1,u_2)\in B_{K_3}$ similarly as before we have
\begin{equation}
\begin{split}
J(u_1,u_2)&=\frac{1}{2}\int_{\mathbb{R}^3} |\nabla u_1|^2+|\nabla u_2|^2-\frac{1}{4}\int \mu_1u_1^4+\mu_2 u_2^4+2\beta u_1^2u_2^2-\int_{\mathbb{R}^3} \kappa(x)u_1u_2\\
&\geq K_3-\frac{C_{a_1,a_2}}{4}(2K_3)^{\frac{3}{2}}-|\kappa(x)|_{\infty}a_1a_2,\nonumber
\end{split}
\end{equation}
then if we take $K_3=\frac{8}{9C_{a_1,a_2}^2}$ as the maximum point of
$$ x-\frac{C_{a_1,a_2}}{4}(2x)^{\frac{3}{2}},$$
 assume
$$|\kappa(x)|_{\infty}<\frac{K_3-\frac{C_{a_1,a_2}}{4}(2K_3)^{\frac{3}{2}}}{a_1a_2}=\frac{8}{27C_{a_1,a_2}^2a_1a_2}=:C_3.$$
then we get \eqref{yun1-19-3}.
\end{proof}
\end{lemma}
From above lemmas let $K=\min\{K_1,K_2,K_3\}=\frac{2}{9C_{a_1,a_2}^2}$ then we have
\begin{lemma}
There exists $K>0$ such that \eqref{yun1-19-1}, \eqref{yun1-19-2}, \eqref{yun1-19-3} hold when
\begin{equation}\label{yun1-19-4}
|\kappa(x)|_{\infty}<\min\{\frac{\frac{1}{2}K_-\frac{C_{a_1,a_2}}{4}(2K)^{\frac{3}{2}}}{2a_1a_2},\frac{K-\frac{C_{a_1,a_2}}{4}(2K)^{\frac{3}{2}}}{a_1a_2}\}=\frac{1}{54C_{a_1,a_2}^2a_1a_2}.
\end{equation}
\end{lemma}\fi
We fix a point $(v_1,v_2)\in A_{K_1}$ both nonnegative, and we try to find a point $(w_1,w_2)$ such that $J(w_1,w_2)$ is negative enough, and $\int_{\mathbb{R}^3} |\nabla w_1|^2+|\nabla w_2|^2$ is large enough. Then any path from $(v_1,v_2)$ to $(w_1,w_2)$ must pass through $B_{K_2}$, so we get a mountain pass structure on manifold $S_1\times S_2$. To do this, we use the translation which was firstly mentioned in [11],
$$
s\star u:=e^{\frac{3s}{2}}u(e^sx),
$$
by direct calculation we have
$$|s\star u|_2^2=|u|_2^2,$$
and
$$|\nabla(s\star u)|_2^2=e^{2s}|\nabla u|_2^2.$$
Moreover, we have
\begin{equation}
\begin{split}
&J(s\star v_1,s\star v_2)\\
={}&\frac{e^{2s}}{2}\int_{\mathbb{R}^3} |\nabla v_1|^2+|\nabla v_2|^2-\frac{e^{3s}}{4}\int_{\mathbb{R}^3} \mu_1 v_1^4+\mu_2 v_2^4+2\beta v_1^2v_2^2-\int_{\mathbb{R}^3} \kappa(x)(s\star v_1)(s\star v_2)\\
\leq{}& \frac{e^{2s}}{2}\int_{\mathbb{R}^3} |\nabla v_1|^2+|\nabla v_2|^2-\frac{e^{3s}}{4}\int_{\mathbb{R}^3} (\mu_1 v_1^4+\mu_2 v_2^4+2\beta v_1^2v_2^2)+|\kappa(x)|_{\infty}a_1a_2.\nonumber
\end{split}
\end{equation}
If $s$ is large enough, then we have $J(s\star v_1,s\star v_2)<-C_1$, where $C_1$ is defined in \eqref{yun1-19-1}, and we take $(w_1,w_2):=(s\star v_1,s\star v_2)$.

Then we can get a mountain pass structure of $J$ on manifold $S_1\times S_2$
\begin{equation}
\Gamma:=\{\gamma(t)=(\gamma_1(t),\gamma_2(t)):\gamma(0)=(v_1,v_2),\gamma(1)=(w_1,w_2)\},
\end{equation}
and the mountain pass value is
\begin{equation}\label{eq3-26-2}
c:=\inf _{\gamma \in \Gamma}\sup_{t\in [0,1]}J(\gamma(t))\geq \inf_{B_{K_2}}J(u_1,u_2)>0.
\end{equation}
In order to obtain the boundedness of the PS sequence at mountain pass value $c$ we use the following notations
\begin{equation}
\tilde{J}(s,u_1,u_2):=J(s\star u_1,s\star u_2)=\tilde {J}(0,s\star u_1,s\star u_2),
\end{equation}
the corresponding minimax structure of $\tilde J$ on $\mathbb{R}\times S_1\times S_2$ as follows
$$\tilde {\Gamma}:=\{\tilde {\gamma}(t)=(s(t),\gamma_1(t),\gamma_2(t)):\tilde\gamma(0)=(0,v_1,v_2),\tilde\gamma(1)=(0,w_1,w_2)\},$$
and it minimax value is
$$\tilde c=\inf_{\tilde{\gamma} \in \tilde{\Gamma}}\sup_{t\in [0,1]}\tilde J(\tilde \gamma(t)).$$
\par
First we claim that $\tilde c=c$.
\par
In fact, from $\tilde \Gamma \supset\Gamma$  we have $\tilde c\leq c$. On the other hand, for any $$\tilde\gamma(t)=(s(t),\gamma_1(t),\gamma_2(t)),$$
by definition we have
$$\tilde J(\tilde\gamma(t))=J(s(t)\star\gamma(t)),$$
and $s(t)\star \gamma (t)\in \Gamma$ is obvious, then
$$\sup_{t\in[0,1]}\tilde J(\tilde\gamma(t))\geq \inf _{\gamma \in \Gamma}\sup_{t\in [0,1]}J(\gamma(t)),$$
by definition of $\tilde c$ we have $\tilde c\geq c$, then $\tilde c=c$. Because
$$\tilde{J}(s,u_1,u_2)=\tilde {J}(0,s\star u_1,s\star u_2),$$
we take a sequence $\tilde \gamma_n=(0,\gamma_{1,n},\gamma_{2,n})\in\tilde \Gamma$
 such that
$$c=\lim_{n\rightarrow\infty}\sup_{t\in[0,1]}\tilde J(\tilde \gamma_n(t)).$$ Moreover, using the fact that $\kappa(x)>0$ we have
$$\tilde J(s,|u_1|,|u_2|)\leq \tilde J(s,u_1,u_2),$$
then we can assume $\gamma_{1,n}$, $\gamma_{2,n}\geq 0$.
  By Theorem 3.2 in \cite{GH-1} (it is easy to check the conditions of Theorem 3.2 in \cite{GH-1} are satisfied by Lemma \ref{le3-25-7}) we can get a PS sequence $(s_n,\tilde u_{1,n},\tilde u_{2,n})$ of $\tilde J$ on $\mathbb{R}\times S_1\times S_2$ at level c. Moreover,
$$
\lim_{n\rightarrow\infty}|s_n|+dist_{H^1_r}((\tilde u_{1,n},\tilde u_{2,n}),(\gamma_{1,n},\gamma_{2,n}))=0.
$$
So we have $s_n\rightarrow 0$ and $\tilde u_{1,n}^{-}$, $\tilde u_{2,n}^{-}\rightarrow 0$ in $H^1_r$, where  and in the following $u^-(x):=\min\{u(x),0\}$ and $u^+(x)=\max\{u(x),0\}$, then by taking $$(u_{1,n},u_{2,n}):=(s_n\star \tilde u_{1,n},s_n\star \tilde u_{2,n}),$$ we have the following lemma.
\begin{lemma}\label{le3-26-1}
$(u_{1,n},u_{2,n})$ is a PS sequence of $J(u_1,u_2)$ at level $c$ on $S_1\times S_2$.
\end{lemma}
\begin{proof}
First we know that $(s_n,\tilde u_{1,n},\tilde u_{2,n})$ is a PS sequence of $\tilde J(s,u_1,u_2)$, then for any $(\phi_1,\phi_2)\in H_r^1\times H_r^1$ we have
\begin{equation}
\begin{split}
&\tilde J'_{\textbf{u}}(s_n,\tilde u_{1,n},\tilde u_{2,n})(\phi_1,\phi_2)\\
={}&e^{2s_n}\int_{\mathbb{R}^3}\nabla\tilde u_{1,n}\cdot\nabla\phi_1+\nabla \tilde u_{2,n}\cdot\nabla \phi_2\\
&-e^{3s_n}\int_{\mathbb{R}^3} \mu_1\tilde u_{1,n}^3\phi_1+\mu_2\tilde u_{2,n}^3+\beta \tilde u_{1,n}\phi_1\tilde u_{2,n}^2+\beta\tilde u_{1,n}^2\tilde u_{2,n}\phi_2\\
&-\int_{\mathbb{R}^3}\kappa(e^{-s_n}x)\tilde u_{1,n}\phi_2-\int_{\mathbb{R}^3}\kappa(e^{-s_n}x)\tilde u_{2,n}\phi_1\\
={}&\int_{\mathbb{R}^3} \nabla u_{1,n}\cdot\nabla (s_n\star \phi_1)+\nabla u_{2,n}\cdot\nabla (s_n\star \phi_2)\\
&-\int_{\mathbb{R}^3} \mu_1 u_{1,n}^3(s_n\star \phi_1)+\mu_2 u_{2,n}^3(s_n\star \phi_1)+\beta u_{1,n}^2 u_{2,n}(s_n\star \phi_2)+\beta u_{1,n}u_{2,n}^2(s_n\star \phi_1)\\
&-\int_{\mathbb{R}^3}\kappa(x)u_{1,n}(s_n\star \phi_2)-\int_{\mathbb{R}^3}\kappa(x)u_{2,n}(s_n\star \phi_1)\\
={}&J'(u_{1,n},u_{2,n})(s_n\star \phi_1,s_n\star \phi_2),\nonumber
\end{split}
\end{equation}
where $\textbf{u}=(u_1,u_2)$. Notice that $-s\star (s\star \phi)=\phi$, $\forall s\in \mathbb{R}$, we have

$$\tilde J'_{\textbf{u}}(s_n,\tilde u_{1,n},\tilde u_{2,n})(-s_n\star\phi_1,-s_n\star\phi_2)= J'(u_{1,n},u_{2,n})( \phi_1, \phi_2).$$
It is obvious that $(\phi_1,\phi_2)\in T_{(u_{1,n},u_{2,n})}S_1\times S_2$ if and only if $(-s_n\star \phi_1,-s_n\star\phi_2)\in T_{(\tilde u_{1,n},\tilde u_{2,n})}S_1\times S_2$ see \cite{B-6}. Since $s_n\rightarrow 0$, we have  $-s_n\star \phi_i\rightarrow\phi_i$, $i=1,2$ as $n\rightarrow \infty$ in $H^1_r$, then for $n$ large enough there exist $A_1>0$ and $A_2>0$ such that
\begin{equation}\label{eq3-26-1}
A_1<\frac{\|(\phi_1,\phi_2)\|}{\|(-s_n\star \phi_1,-s_n\star \phi_2)\|}<A_2,
\end{equation}
where $(\phi_1,\phi_2)\neq (0,0)$. Let $\|\cdot\|_\star$ be the norm of the cotangent space $(T_{(u_{1},u_{2})}S_1\times S_2)^\star$. Thus for any  $(\phi_1,\phi_2)\in T_{(u_{1,n},u_{2,n})}S_1\times S_2$ and $(\phi_1,\phi_2)\neq (0,0)$, we have
\begin{equation}
\begin{split}
| J|'_{S_1\times S_2}(u_{1,n},u_{2,n})\frac{(\phi_1,\phi_2)}{\|(-s_n\star \phi_1,-s_n\star \phi_2)\|}|&\leq \|(\tilde J|_{S_1\times S_2} )'_{\textbf u}(s_n,\tilde u_{1,n},\tilde u_{2,n})\|_{\star}\rightarrow 0,\nonumber
\end{split}
\end{equation}
 as $n\rightarrow \infty$. Take the supremum both side and notice \eqref{eq3-26-1}, we have
$$A_1\| J|'_{S_1\times S_2}(u_{1,n},u_{2,n})\|_{\star} \leq \|(\tilde J|_{S_1\times S_2} )'_{\textbf u}(s_n,\tilde u_{1,n},\tilde u_{2,n})\|_{\star}\rightarrow 0,~\hbox{as}~ n\rightarrow \infty.$$
 From the fact that $A_1>0$ we have $$\| J|'_{S_1\times S_2}(u_{1,n},u_{2,n})\|_{\star}\rightarrow 0,~\hbox{as}~ n\rightarrow \infty.$$ On the other hand we have
$$J(u_{1,n},u_{2,n})=\tilde J(s_n,\tilde u_{1,n},\tilde u_{2,n})\rightarrow c,~\hbox{as}~ n\rightarrow \infty.$$
 This finishes the proof.
\end{proof}
\begin{lemma}\label{le3-26-2}
If $\kappa(x)$ and $\nabla\kappa(x)\cdot x$ is bounded in $\mathbb{R}^3$, then the \textup{PS} sequence $(u_{1,n},u_{2,n})$ obtained in Lemma \ref{le3-26-1} of $J(u_1,u_2)$ on $S_1\times S_2$ at level $c$ is bounded in $H_r^1\times H_r^1$.
\end{lemma}
\begin{proof}
Since $(s_n,\tilde u_{1,n},\tilde u_{2,n})$ is a PS sequence for $\tilde J$, we have $$\frac{d}{ds}\tilde J(s_n,\tilde u_{1,n},\tilde u_{2,n})\rightarrow 0,$$
i.e.,
\begin{equation}\label{yun2-4-1}
\begin{split}
&\int_{\mathbb{R}^3} |\nabla u_{1,n}|^2+|\nabla u_{2,n}|^2-\frac{3}{4}\int_{\mathbb{R}^3} \mu_1 u_{1,n}^4+\mu_2 u_{2,n}^4+2\beta u_{1,n}^2u_{2,n}^2\\
&+\int_{\mathbb{R}^3} \nabla\kappa(e^{-s_n}x)\cdot e^{-s_n}x\tilde u_{1,n}\tilde u_{2,n}\rightarrow 0.
\end{split}
\end{equation}
On the other hand, notice that $\tilde J(s_n,\tilde u_{1,n},\tilde u_{2,n})=J(u_{1,n},u_{2,n})$, we obtain
\begin{equation}\label{yun2-4-2}
\begin{split}
J(u_{1,n},u_{2,n})=&\frac{1}{2}\int_{\mathbb{R}^3} |\nabla u_{1,n}|^2+|\nabla u_{2,n}|^2-\frac{1}{4}\int_{\mathbb{R}^3} \mu_1 u_{1,n}^4+\mu_2 u_{2,n}^4+2\beta u_{1,n}^2u_{2,n}^2\\
&-\int_{\mathbb{R}^3}\kappa(e^{-s_n}x)\tilde u_{1,n}\tilde u_{2,n}\rightarrow c,
\end{split}
\end{equation}
then use the boundness of $\kappa(x)$, $\nabla \kappa(x)\cdot x$, $(\tilde u_{1,n},\tilde u_{2,n})\in S_1\times S_2$, \eqref{yun2-4-1} and $\eqref{yun2-4-2}$, we can deduce that $\int_{\mathbb{R}^3} |\nabla u_{1,n}|^2+|\nabla u_{2,n}|^2$ is bounded, notice that $(u_{1,n},u_{2,n})\in S_1\times S_2$, we get $( u_{1,n}, u_{2,n})$ is bounded in $H^1_r\times H^1_r$.
\end{proof}
Because $(u_{1,n},u_{2,n})$ is bounded in $H^1_r\times H^1_r$, then there exists $(\bar u_{1},\bar u_{2})\in H^1_r\times H^1_r$ such that
$$(u_{1,n},u_{2,n})\rightharpoonup (\bar u_{1},\bar u_{2}),~\hbox{in}~H^1_r\times H^1_r.$$
\begin{lemma}\label{le3-26-3}
Under the assumptions of Lemma \ref{le3-26-2}, and we assume $\frac{1}{3}\nabla \kappa(x)\cdot x+\kappa(x)\geq 0$, then there exists $C>0$ such that for $n$ large we have
$$|\nabla u_{1,n}|_2^2+|\nabla u_{2,n}|_2^2\geq C.$$
\end{lemma}
\begin{proof}
By \eqref{yun2-4-1} and $\tilde u_{1,n}^{-},\tilde u_{2,n}^{-}\rightarrow 0$ in $H^1_r$, we have
\begin{equation}
\begin{split}
&\int_{\mathbb{R}^3} |\nabla u_{1,n}|^2+|\nabla u_{2,n}|^2-\frac{3}{4}\int_{\mathbb{R}^3} \mu_1 u_{1,n}^4+\mu_2 u_{2,n}^4+2\beta u_{1,n}^2u_{2,n}^2\\
&+\int_{\mathbb{R}^3} \nabla\kappa(e^{-s_n}x)\cdot e^{-s_n}x\tilde u_{1,n}^+\tilde u_{2,n}^+=o(1),\nonumber
\end{split}
\end{equation}
and
\begin{equation}
\begin{split}
\frac{1}{2}&\int_{\mathbb{R}^3} |\nabla u_{1,n}|^2+|\nabla u_{2,n}|-\frac{1}{4}\int_{\mathbb{R}^3} \mu_1 u_{1,n}^4+\mu_2 u_{2,n}^4+2\beta u_{1,n}^2u_{2,n}^2\\
&-\int_{\mathbb{R}^3}\kappa(e^{-s_n}x)\tilde u_{1,n}^+\tilde u_{2,n}^+= c+o(1).\nonumber
\end{split}
\end{equation}
By \eqref{eq3-26-2},  we have $c>0$, thus
\begin{equation}
\begin{split}
c+o(1)&=J(u_{1,n},u_{2,n})\\
&= \frac{1}{6}\int_{\mathbb{R}^3} |\nabla u_{1,n}|^2+|\nabla u_{2,n}|^2-\int_{\mathbb{R}^3} (\frac{1}{3} \nabla\kappa(e^{-s_n}x)\cdot e^{-s_n}x+\kappa(e^{-s_n}x))\tilde u_{1,n}^+\tilde u_{2,n}^+\\
&\leq \frac{1}{6}\int_{\mathbb{R}^3} |\nabla u_{1,n}|^2+|\nabla u_{2,n}|^2,\nonumber
\end{split}
\end{equation}
then for $n$ large enough and take $C=3c$, this finishes the proof.

\end{proof}
Because $(u_{1,n},u_{2,n})$ is a PS sequence of $J$ on $S_1\times S_2$, for any $(\phi_1,\phi_2)\in H^1_r\times H^1_r$, there exists $\lambda_{1,n},\lambda_{2,n}$ such that as $n\rightarrow\infty$
\begin{equation}\label{1-21-1}
\begin{split}
&(J|'_{S_1\times S_2}(u_{1,n},u_{2,n}),(\phi_1,\phi_2))\\
={}&\int_{\mathbb{R}^3} \nabla u_{1,n}\nabla \phi_1+\int_{\mathbb{R}^3} \nabla u_{2,n}\nabla \phi_2-\mu_1\int_{\mathbb{R}^3} u_{1,n}^3\phi_1-\mu_2\int_{\mathbb{R}^3} u_{2,n}^3\phi_2\\
&-\beta\int_{\mathbb{R}^3} u_{1,n}u_{2,n}^2\phi_1-\beta\int_{\mathbb{R}^3} u_{1,n}^2u_{2,n}\phi_2-\int_{\mathbb{R}^3}\kappa(x)u_{1,n}\phi_2\\
&-\int_{\mathbb{R}^3}\kappa(x)u_{2,n}\phi_1-\lambda_{1,n}\int_{\mathbb{R}^3} u_{1,n}\phi_1-\lambda_{2,n}\int_{\mathbb{R}^3} u_{2,n}\phi_2\\
={}&o(\|(\phi_1,\phi_2)\|).
\end{split}
\end{equation}
 From the proof of Theorem \ref{th3-23-1}, taking $p_1=p_2=4$ and $r_1=r_2=2$ we have
\begin{equation}\label{eq3-26-3}
\lambda_{1,n}a_1^2=\int_{\mathbb{R}^3} |\nabla u_{1,n}|^2-\mu_1\int_{\mathbb{R}^3} u_{1,n}^4-\beta \int_{\mathbb{R}^3} u_{1,n}^2u_{2,n}^2-\int_{\mathbb{R}^3} \kappa(x) u_{1,n}u_{2,n},
\end{equation}
\begin{equation}\label{eq3-26-4}
\lambda_{2,n}a_2^2=\int_{\mathbb{R}^3} |\nabla u_{2,n}|^2-\mu_2\int_{\mathbb{R}^3} u_{2,n}^4-\beta \int_{\mathbb{R}^3} u_{1,n}^2u_{2,n}^2-\int_{\mathbb{R}^3} \kappa(x) u_{1,n}u_{2,n},
\end{equation}
then it is easy to deduce that $\{\lambda_{1,n}\}$ and $\{\lambda_{2,n}\}$ are bounded. So we may assume
$$\lambda_{1,n}\rightarrow \bar\lambda_1,$$
$$\lambda_{2,n}\rightarrow \bar\lambda_2,$$
by choosing subsequence if necessary.
\begin{lemma}\label{le3-26-4}
Under the conditions of Lemma \ref{le3-26-3}, assume $\frac{2}{3}\nabla \kappa(x)\cdot x+\kappa(x)\geq 0$ and $\kappa(x)>0$, then at least one of $\bar\lambda_i$, $i=1,2$ is negative.
\end{lemma}
\begin{proof}
Notice that $\tilde u_{1,n}^-\rightarrow 0$, $\tilde u_{2,n}^-\rightarrow 0$ in $H^1_r$, \eqref{eq3-26-3}, \eqref{eq3-26-4} and \eqref{yun2-4-1}, we have
\begin{equation}
\begin{split}
&\bar\lambda_1a_1^2+\bar\lambda_2a_2^2\\
={}&\lambda_{1,n}a_1^2+\lambda_{2,n}a_2^2+o(1)\\
={}&\int_{\mathbb{R}^3} |\nabla u_{1,n}|^2+|\nabla u_{2,n}|^2-\int_{\mathbb{R}^3} \mu_1 u_{1,n}^4+\mu_2 u_{2,n}^4+2\beta u_{1,n}^2u_{2,n}^2-2\int_{\mathbb{R}^3} \kappa(x)u_{1,n}u_{2,n}+o(1)\\
={}&-\frac{1}{3}\int_{\mathbb{R}^3} |\nabla u_{1,n}|^2+|\nabla u_{2,n}|^2-(\int_{\mathbb{R}^3}  (\frac{4}{3}\nabla\kappa(e^{-s_n}x)\cdot e^{-s_n}x+2\kappa(e^{-s_n}x))\tilde u_{1,n}^+\tilde u_{2,n}^+)+o(1)\\
\leq{}& -\frac{1}{3}\int_{\mathbb{R}^3} |\nabla u_{1,n}|^2+|\nabla u_{2,n}|^2+o(1)\\
<{}&-\frac{1}{3}C+o(1),\nonumber
\end{split}
\end{equation}
then one of the $\bar\lambda_1$, $\bar\lambda_2$ is negative.
\end{proof}
\begin{proof}[Proof of Theorem 1.3.] From standard argument we can conclude that $(\bar\lambda_1,\bar\lambda_2,\bar u_1,\bar u_2)$ is a solution of the system
\begin{equation}\label{eq:diricichlet}
\begin{cases}
-\Delta u_1-\lambda_1 u_1=\mu_1 u_1^3+\beta u_1u_2^2+\kappa (x)u_2\ \textup{in}\ \mathbb{R}^3,\\
-\Delta u_2-\lambda_2 u_2=\mu_2 u_2^3+\beta u_1^2 u_2+\kappa (x)u_1\ \textup{in}\ \mathbb{R}^3,\\ u_1\in H^1(\mathbb{R}^3), u_2\in H^1(\mathbb{R}^3),
\end{cases}
\end{equation}
where $\bar u_1$ and $\bar u_2$ are nonnegative, then we just need to prove $(u_{1,n},u_{2,n})\rightarrow (\bar u_1,\bar u_2)$ strongly in $H^1_r\times H^1_r$. By \eqref{1-21-1} together with compact embedding and $\lambda_{1,n}\rightarrow \bar\lambda_1$ we have
\begin{equation}
\begin{split}
o(1)&=(J'(u_{1,n},u_{2,n})-J'(\bar u_{1},\bar u_{2}),(u_{1,n}-\bar u_{1},0))-\bar\lambda_1\int_{\mathbb{R}^3}(u_{1,n}-\bar u_{1})^2\\
&=\int_{\mathbb{R}^3}|\nabla (u_{1,n}-\bar u_{1})|^2-\bar\lambda_1(u_{1,n}-\bar u_{1})^2+o(1).\nonumber
\end{split}
\end{equation}
As the proof of Lemma \ref{le3-25-2} we just need to show that
$\bar\lambda_1,\bar\lambda_2 <0$ to obtain the strong convergence. When $\kappa(x)\in L^{p}(\mathbb{R}^3)$ for some $\frac{3}{2}<p<\infty$. By Lemma \ref{le3-26-4} without loss of generality we may assume $\bar\lambda_1<0$ if $\bar\lambda_2\geq 0$, then we have
$$-\Delta \bar u_2=\bar\lambda_2 \bar u_2+\mu_1 \bar u_2^3+\beta \bar u_1^2 \bar u_2+\kappa(x) \bar u_1\geq 0,$$
 by Lemma \ref{3-24-2}, we have $ \bar u_2\equiv 0$. But $u_{1,n}\rightarrow \bar u_1$ in $H^1_r$, thus $|\bar u_1|_2^2=a_1^2$, so $\bar u_1\not\equiv 0$ and $(\bar u_1,0)$ can not be the solution of \eqref{yun1-18-5}, therefore $\bar\lambda_2<0$. From maximum principle we can get $\bar u_1,\bar u_2>0$, then we finish the proof of Theorem  \ref{zh3-23}.
\end{proof}
\begin{remark}
Because $\kappa(x)>0$, from $\frac{2}{3}\nabla \kappa(x)\cdot x+\kappa(x)\geq 0$ we can deduce that $\frac{1}{3}\nabla \kappa(x)\cdot x+\kappa(x)\geq 0$, so we just need assume $\frac{2}{3}\nabla \kappa(x)\cdot x+\kappa(x)\geq 0$ in Theorem \ref{zh3-23}.
\end{remark}

Moreover, from the proof of Theorem \ref{zh3-23} we can get a stronger theorem. First we introduce a new condition

($K_1$) There exist $T_1,T_2>0$ such that $$T_1+2T_2<\frac{1}{108C_{a_1,a_2}^2a_1a_2}\leq \frac{\inf_{B_{K_2}}J}{2a_1a_2}\leq \frac{c}{2a_1a_2},$$ where $c$ is the mountain pass value defined in \eqref{eq3-26-2} and $$\nabla \kappa(x)\cdot x\geq \max\{-3T_1,-\frac{3}{2}T_2\}.$$
\begin{thm}\label{th3-26-1}
Assume that $\kappa(x)>0$, $\kappa(x)=\kappa(|x|)$, $\kappa(x)\in L^p(\mathbb{R}^3)\cap L^{\infty}(\mathbb{R}^3)$ where $\frac{3}{2}<p<\infty$, $\nabla \kappa(x)\cdot x$ is bounded. Moreover, if $\kappa(x)$ satisfies ($K_1$) and $|\kappa(x)|_{\infty}<\frac{5}{18C_{a_1,a_2}^2}$, then  system \eqref{yun1-18-5}-\eqref{yun1-18-2} has a solution  $(\bar\lambda_1,\bar\lambda_2,\bar u_1,\bar u_2)$.  Moreover, $\bar u_1>0$, $\bar u_2>0$ are radial and $\bar\lambda_1<0$, $\bar\lambda_2<0$.
\end{thm}
\begin{proof}
We just need to show that one of the limits of $\{\lambda_{1,n}\}$ and $\{\lambda_{2,n}\}$ is negative. By $(K_1)$ and $\kappa(x)>0$ we have
$$\frac{1}{3}\nabla \kappa(x)\cdot x+\kappa(x)\geq -T_1,$$
$$\frac{2}{3}\nabla \kappa(x)\cdot x+\kappa(x)\geq -T_2,$$
and from above two equations we have
\begin{equation}
\begin{split}
c+o(1)&=J(u_{1,n},u_{2,n})\\
&\leq \frac{1}{6}\int_{\mathbb{R}^3} |\nabla u_{1,n}|^2+|\nabla u_{2,n}|^2-\int_{\mathbb{R}^3}(\frac{1}{3} \nabla\kappa(e^{-s_n}x)\cdot e^{-s_n}x+\kappa(e^{-s_n}x))\tilde u_{1,n}^+\tilde u_{2,n}^+\\
&\leq \frac{1}{6}\int_{\mathbb{R}^3} |\nabla u_{1,n}|^2+|\nabla u_{2,n}|^2\nonumber+T_1a_1a_2,
\end{split}
\end{equation}
then we have for $n$ large $\frac{1}{6}\int_{\mathbb{R}^3} |\nabla u_{1,n}|^2+|\nabla u_{2,n}|^2>\frac{c}{3}$.
\begin{equation}
\begin{split}
&\lambda_1a_1^2+\lambda_2a_2^2\\
={}&\lambda_{1,n}a_1^2+\lambda_{2,n}a_2^2+o(1)\\
={}&\int_{\mathbb{R}^3} |\nabla u_{1,n}|^2+|\nabla u_{2,n}|^2-\int_{\mathbb{R}^3} \mu_1 u_{1,n}^4+\mu_2 u_{2,n}^4+2\beta u_{1,n}^2u_{2,n}^2-2\int_{\mathbb{R}^3} \kappa(x)u_{1,n}u_{2,n}+o(1)\\
={}&-\frac{1}{3}\int_{\mathbb{R}^3} |\nabla u_{1,n}|^2+|\nabla u_{2,n}|^2-(\int_{\mathbb{R}^3}  (\frac{4}{3}\nabla\kappa(e^{-s_n}x)\cdot e^{-s_n}x+2\kappa(e^{-s_n}x))\tilde u_{1,n}^+\tilde u_{2,n}^+)+o(1)\\
\leq{}& -\frac{1}{3}\int_{\mathbb{R}^3} |\nabla u_{1,n}|^2+|\nabla u_{2,n}|^2+2T_2a_1a_2+o(1)\\
<{}&-\frac{2}{3}c+\frac{1}{2}c+o(1)\nonumber\\
={}&-\frac{1}{6}c+o(1),
\end{split}
\end{equation}
we know that one of $\bar\lambda_i$, $i=1,2$ is negative, following the steps of the proof of Theorem \ref{zh3-23}, we can get Theorem \ref{th3-26-1}.
\end{proof}
\begin{remark}
From above arguments we can relax the condition $\kappa(x)>0$ to $\kappa(x)\geq 0$ and $\kappa(x)\not\equiv 0$ or $\kappa(x)<0$ to $\kappa(x)\leq 0$ and $\kappa(x)\not\equiv 0$ in Theorem \ref{th3-23-1}, Theorem \ref{th3-23-2}, Theorem \ref{zh3-23}  and Theorem \ref{th3-26-1}. In fact, if $\bar u_2\equiv 0$, then $\kappa(x)\bar u_1\equiv 0$ and $\bar u_1$ solves the equation
$$-\Delta u-\lambda_1 u=\mu_1 u^3\ \textup{in}\ \mathbb{R}^3,$$
where $\lambda_1<0$ (see the proof of Theorem \ref{zh3-23}). It has a unique positive solution up to  translations. So $\kappa(x)\bar u_1\equiv 0$ is impossible, thus $\bar u_2\not\equiv 0$.
\end{remark}


\begin{thebibliography}{0}
\bibitem{B-1}T. Bartsch, L. Jeanjean, Normalized solutions for nonlinear Schr\"{o}dinger systems,  Proc. Roy. Soc. Edinburgh Sect. A \textbf{148} (2018), no. 2, 225-242.
\bibitem{B-2}%(MR3009665)
T. Bartsch, S. de Valeriola, Normalized solutions of nonlinear Schr\"{o}dinger equations. Arch. Math. (Basel) \textbf{100} (2013), no. 1, 75-83.

%\bibitem{B-3}%(MR3777573)
%T. Bartsch, L. Jeanjean, Normalized solutions for nonlinear Schr\"{o}dinger systems. Proc. Roy. Soc. Edinburgh Sect. A \textbf{148} (2018), no. 2, 225-242.

\bibitem{B-4}%(MR3539467)
T. Bartsch, L. Jeanjean, N. Soave, Normalized solutions for a system of coupled cubic Schr\"{o}dinger equations on $\mathbb{R}^3$. J. Math. Pures Appl. \textbf{106} (2016), no. 4, 583-614.

\bibitem{B-5}%(MR3639521)
T. Bartsch, N. Soave, A natural constraint approach to normalized solutions of nonlinear Schr\"{o}dinger equations and systems. J. Funct. Anal. \textbf{272} (2017), no. 12, 4998-5037.

\bibitem{B-6}%(MR3895385)
T. Bartsch, N. Soave, Multiple normalized solutions for a competing system of Schr\"{o}dinger equations. Calc. Var. Partial Differential Equations \textbf{58} (2019), no. 1, Art. 22, 24 pp.

\bibitem{B-7}%(MR3802492)
T. Bartsch, N. Soave, Correction to: "A natural constraint approach to normalized solutions of nonlinear Schr\"{o}dinger equations and systems'' [J. Funct. Anal. 272 (12) (2017) 4998-5037] [MR3639521]. J. Funct. Anal.\textbf{ 275} (2018), no. 2, 516-521.
\bibitem{N-1}
Benedetta Noris, Hugo Tavares and Gianmaria Verzini, Normalized solutions for nonlinear
Schr\"{o}dinger systems on bounded domains,Nonlinearity 32 (2019), 1044-1072
\bibitem{D-1} Deconinck B., Kevrekidis P. G., Nistazakis H. E., Frantzeskakis D.J., Linearly coupled
Bose-Einstein condensates: From Rabi oscillations and quasiperiodic solutions to
oscillating domain walls and spiral waves, Physical Review A, 70(6) (2004), 063605.
\bibitem{GH-1}
N. Ghoussoub. Duality and perturbation methods in critical point theory, Cambridge Tracts
in Mathematics, vol. \textbf{107}, Cambridge University Press, 1993.
\bibitem{G-Z}
Tianxiang Gou, Zhitao Zhang, Normalized solutions to the Chern-Simons-Schr\"{o}dinger system. J. Funct. Anal. 280 (2021), no. 5, 108894, 65pp.
\bibitem{IK-1}
N. Ikoma. Compactness of minimizing sequences in nonlinear Schr\"{o}dinger systems under
multiconstraint conditions. Adv. Nonlin. Studies \textbf{14} (2014), 115-136.

\bibitem{LJ-1}%(MR1430506)
L. Jeanjean, Existence of solutions with prescribed norm for semilinear elliptic equations. Nonlinear Anal. \textbf{28} (1997), no. 10, 1633-1659.

\bibitem{LZ-11}
K. Li, Z. T. Zhang, Existence of solutions for a Schr\"{o}dinger system with linear and nonlinear couplings, J. Math. Phys.
\textbf{57} (8) (2016), 081504, 17 pp.

\bibitem{L-Z}
Haijun Luo,  Zhitao Zhang, Normalized solutions to the fractional Schr\"{o}dinger equations with combined nonlinearities. Calc. Var. Partial Differential Equations 59 (2020), no. 4, Paper No. 143, 35 pp.

\bibitem{LZ-2}
H. J. Luo, Z. T. Zhang, Existence and nonexistence of bound state solutions for
Schr\"{o}dinger systems with linear and nonlinear couplings, J. Math. Anal. Appl. \textbf{475} (2019), 350-363
\bibitem{NTV14}%(MR3318740)
B. Noris, H. Tavares, G. Verzini, Existence and orbital stability of the ground states with prescribed mass for the $L^2$-critical and supercritical NLS on bounded domains. Anal. PDE \textbf{7} (2014), no. 8, 1807-1838.

\bibitem{NTV15}%(MR3393268)
B. Noris, H. Tavares, G. Verzini, Stable solitary waves with prescribed $L^2$-mass for the cubic Schr\"{o}dinger system with trapping potentials. Discrete Contin. Dyn. Syst. \textbf{35} (2015), no. 12, 6085-6112.

\bibitem{NTV18}%(MR3918087)
B. Noris, H. Tavares, G. Verzini, Normalized solutions for nonlinear Schr\"{o}dinger systems on bounded domains.  Nonlinearity \textbf{32} (2019), no. 3, 1044-1072.
\bibitem{PV17}%(MR3689156)
D. Pierotti, G. Verzini, Normalized bound states for the nonlinear Schr\"{o}dinger equation in bounded domains. Calc. Var. Partial Differential Equations \textbf{56} (2017), no. 5, Art. 133, 27 pp.

\bibitem{TZ-1}
 R.S. Tian, Z.T. Zhang, Existence and bifurcation of solutions for a double coupled system of Schr\"{o}dinger equations, Sci.
China Math. \textbf{58} (8) (2015), 1607-1620.
\bibitem{ZZ}
X.~Zhang and Z.~Zhang,
\newblock Distribution of positive solutions to {Schr\"{o}dinger} systems with
  linear and nonlinear couplings,
\newblock {\em J. Fixed Point Theory Appl.}, 22(2) (2020), 33,21pp.
\bibitem{Z}
Z.T. Zhang,
\newblock Variational, topological, and partial order methods with their
  applications,
\newblock Springer, Heidelberg, 2013.
\bibitem{ZhZh}
X.P. Zhu, H.S. Zhou, Bifurcation from the essential spectrum of superlinear elliptic equations. Appl. Anal. 28 (1988), no. 1, 51-66.
\end{thebibliography}
\end{document}